\documentclass{amsart}

\usepackage{amssymb}
\usepackage{amsmath}
\usepackage{amsthm}

\theoremstyle{plain}

\newtheorem{thm}{Theorem}
\newtheorem{cor}[thm]{Corollary}
\newtheorem{lem}[thm]{Lemma}

{
  \let\olduparrow=\uparrow
  \xdef\uparrow{\mathpunct{\olduparrow}}
}

{
  \let\olddownarrow=\downarrow
  \xdef\downarrow{\mathpunct{\olddownarrow}}
}

\DeclareMathOperator{\Con}{Con}
\DeclareMathOperator{\At}{At}

\begin{document}

\title[Preprint]{MacNeille completion and profinite completion can coincide
  on finitely generated modal algebras}

\author[Submitted to Algebra Universalis]{Jacob Vosmaer} \email{contact@jacobvosmaer.nl}
\address{Plantage Muidergracht 24\\1018 TV Amsterdam\\The Netherlands}

\subjclass[2000]{Primary 06E25; Secondary 06B23, 03B45, 22A30}

\keywords{modal algebra, MacNeille completion, profinite completion}

\thanks{This research was supported by VICI grant 639.073.501 of the
  Netherlands Organization for Scientific Research (NWO)}

\date{April 28, 2008}
\maketitle

\begin{abstract}
  Following Bezhanishvili \& Vosmaer, we confirm a conjecture of Yde Venema by
  piecing together results from various authors. Specifically, we show
  that if $\mathbb{A}$ is a residually finite, finitely generated
  modal algebra such that $\operatorname{HSP}(\mathbb{A})$ has
  equationally definable principal congruences, then the profinite
  completion of $\mathbb{A}$ is isomorphic to its MacNeille
  completion, and $\Diamond$ is smooth. Specific examples of such
  modal algebras are the free $\mathbf{K4}$-algebra and the free
  $\mathbf{PDL}$-algebra.
\end{abstract}

\section{Introduction}
In this paper we compare two mathematical constructions applied to
modal algebras. The first is the MacNeille completion, which is an
order-theoretic generalization of the construction of the reals from
the rationals using Dedekind cuts \cite{MacNeille1937}. It has been
applied in logic to e.g.~prove the completeness of predicate calculi
\cite{RaSi1950}. The second is the profinite completion, which is a
universal algebraic construction, transforming an algebra into a
topological algebra endowed with a Stone (compact, Hausdorff,
zero-dimensional) topology. This construction stems from Galois theory
\cite{RiZa2000}, but has more recently also been connected with
lattice completions \cite{BGMM2005, Harding2005, Vosmaer2006, BV2007}.

This paper is a companion piece to \cite{BV2007}. In that paper,
parallel versions of our Theorems \ref{ProfMacN} and \ref{Thm:FinGen}
arise in a study of the connections between different completions of
Heyting algebras, using Esakia duality. In light of the topological
character of the profinite completion, in the present paper we will
present topological algebra proofs instead. This establishes a strong
connection with the body of work on canonicity \cite{GehJon2004} and
MacNeille canonicity \cite{TheuVen2005}. Another advantage of our
present perspective is that we can show how the two main Theorems
pivot around an interaction between principal lattice filters and
principal algebra congruences, which are in a 1-1 correspondence for
Heyting algebras, but not for modal algebras. Finally, we will briefly
mention some of the connections of our results to modal logic.

The author would like to thank Yde Venema, who suggested that Theorem
\ref{Thm:FinGen} might be true. Additionally, the author is grateful
to the Editor and the Referee for their criticisms and suggestions.

\section{Completions and topologies} \label{CompTop}

Let $\mathbb{B}= \langle B; \wedge,\vee,\neg ,0,1 \rangle$ be a
Boolean algebra. Given $b\in B$ we write $b\downarrow = \{a\in B \mid
a\leq b\}$ ($b\uparrow$ is defined dually).  We say $S\subseteq B$ is
\emph{join-dense} in $\mathbb{B}$ iff for every $a\in B$, $a= \bigvee
( a \downarrow \cap S)$ (\emph{meet-density} is defined dually). A
\emph{completion} of a lattice $\mathbb{B}$ is a pair $( m,
\mathbb{C})$, where $m\colon \mathbb{B} \hookrightarrow \mathbb{C}$ is
a lattice embedding into a complete lattice $\mathbb{C}$. Completions
$(m, \mathbb{C})$ and $(k, \mathbb{D})$ of $\mathbb{B}$ are isomorphic
if $gm=k$ for some lattice isomorphism $g\colon \mathbb{C} \rightarrow
\mathbb{D}$.  If $(m,\mathbb{C})$ is a completion of $\mathbb{B}$, let
$\rho_{\mathbb{B}}$ be the topology on $C$ generated by basis $\{
[m(a),m(b)] \mid a,b \in B\}$ (where $[x,y]=\{z \in C \mid x\leq z
\leq y\}$). By $\gamma_{\mathbb{B}}^{\downarrow},
\gamma_{\mathbb{B}}^{\uparrow}$ and $\gamma_{\mathbb{B}}$ we denote
the Scott topology, the dual Scott topology, and the biScott topology
on $\mathbb{B}$ respectively. Let $\At \mathbb{B}$ be the (possibly
empty) set of atoms of $\mathbb{B}$, and let $\At_{\omega} \mathbb{B}$
be the set of all finite joins of atoms of $\mathbb{B}$. Then
$\iota_{\mathbb{B}}$ is the topology generated by the basis $\{ [a,
\neg b] \mid a,b \in \At_{\omega} \mathbb{B}\}$. By \cite[Section
2]{GehJon2004}, $\iota_{\mathbb{B}} = \gamma_{\mathbb{B}}$ if
$\mathbb{B}$ is complete and atomic.

The \emph{MacNeille completion} \cite{BaBr1967} of a Boolean algebra
$\mathbb{B}$ is defined up to isomorphism as a completion $(m,
\mathbb{C})$ such that $m[B]$ is join-dense in $\mathbb{C}$ (by
\cite[Theorem V-27]{Birkhoff1967} $\mathbb{C}$ is then also a Boolean
algebra). We denote the MacNeille completion of $\mathbb{B}$ by
$\mathbb{\bar B}$.  Alternatively \cite[Theorem 4.5]{TheuVen2005},
$\mathbb{\bar B}$ can be characterized up to isomorphism as a
completion $(m,\mathbb{C})$ of $\mathbb{B}$ such that $\langle C,
\rho_{\mathbb{B}}\rangle$ is Hausdorff.  If $f\colon \mathbb{B}
\rightarrow \mathbb{C}$ is an order-preserving map between Boolean
algebras, then $f^{\circ} \colon \mathbb{\bar B} \rightarrow
\mathbb{\bar C}$, defined by $f^{\circ} \colon x \mapsto \bigvee \{
f(a) \mid m_{\mathbb{B}}(a) \leq x \}$, is the \emph{lower extension}
of $f$. The \emph{upper extension} $f^{\bullet}$ is defined dually.
Alternatively \cite[Section 5]{TheuVen2005}, $f^{\circ}$ is the
(pointwise) largest $(\rho_{\mathbb{B}}, \gamma_{\mathbb{\bar
    C}}^{\downarrow})$-continuous extension of $f$, and $f^{\bullet}$
is the smallest $(\rho_{\mathbb{B}}, \gamma_{\mathbb{\bar
    C}}^{\uparrow})$-continuous extension of $f$. We say $f$ is
\emph{smooth} if $f^{\circ}= f^{\bullet}$.

Given a modal algebra $\mathbb{A}= \langle A; \Diamond \rangle$, let
$\Phi_{\mathbb{A}} := \{ \theta\in \Con \mathbb{A} \mid
\mathbb{A}/\theta \text{ is finite}\}$. We say $\mathbb{A}$ is
\emph{residually finite} if for all $a,b\in \mathbb{A}$ with $a \neq
b$, there exists $\theta\in \Phi_{\mathbb{A}}$ such that $a/\theta
\neq b/\theta$.  The inverse system $\langle
\{\mathbb{A}/\theta\}_{\theta\in \Phi_{\mathbb{A}}},
f_{\theta\psi}\rangle$, where $f_{\theta\psi} \colon \mathbb{A}/\theta
\twoheadrightarrow \mathbb{A}/\psi$ (for all $\theta,\psi\in
\Phi_{\mathbb{A}}$ such that $\theta \subseteq \psi$) is defined by
$f_{\theta\psi} \colon a/\theta \mapsto a/\psi$, has a projective
limit
\[ \mathbb{\hat A}=\big\{ \alpha \in
\textstyle{\prod_{\Phi_{\mathbb{A}}}}\mathbb{A}/\theta \mid \forall
\theta,\psi \in \Phi_{\mathbb{A}} \text{ with } \theta\subseteq\psi,
\text{ if } \alpha(\theta)=a/\theta \text{ then } \alpha(\psi)=a/\psi
\big\}. \] The map $\mu \colon \mathbb{A} \rightarrow \mathbb{\hat
  A}$, defined by $\mu \colon a \mapsto (a/\theta)_{\theta\in
  \Phi_{\mathbb{A}}}$, is a modal algebra homomorphism which is
injective iff $\mathbb{A}$ is residually finite.  We call
$\mathbb{\hat A}$ the \emph{profinite completion of $\mathbb{A}$}
\cite{RiZa2000}.  Since $\mathbb{\hat A}$ is a complete lattice
\cite{Harding2005}, it follows that $(\mu, \mathbb{\hat A})$ is a
completion of $\mathbb{A}$ iff $\mathbb{A}$ is residually finite.  If
we define the discrete topology on each $\mathbb{A}/\theta$,
$\mathbb{\hat A}$ inherits a topology $\tau_{\mathbb{\hat A}}$ as a
closed subspace of the product $\prod_{\Phi_{\mathbb{A}}}
\mathbb{A}/\theta$. Now $\langle \hat A, \tau_{ \mathbb{ \hat A}}
\rangle$ is a Stone space \cite[Section 2]{BGMM2005}, and in
particular $\widehat{\Diamond} \colon \mathbb{\hat A} \rightarrow
\mathbb{\hat A}$ is $(\tau_{\mathbb{\hat A}}, \tau_{ \mathbb{\hat A}}
)$-continuous \cite{Bana1971}.

\begin{lem} \label{Tops} 
  If $\mathbb{A}$ is a Boolean algebra expansion, then
  $\tau_{\mathbb{\hat A}} = \iota_{\mathbb{\hat A}} =
  \gamma_{\mathbb{\hat A}}$.
\end{lem}

\begin{proof}
  Since $\langle\mathbb{\hat A}, \tau_{\mathbb{\hat A}} \rangle$ is a
  compact Hausdorff topological lattice, it follows by \cite[Corollary
  VII-2.3]{CCL1980} that $\tau_{\mathbb{\hat A}}=\gamma_{\mathbb{\hat
      A}}$. Since $\mathbb{\hat A}$ is also a complete, atomic Boolean
  algebra \cite{BGMM2005, Vosmaer2006}, we know that
  $\iota_{\mathbb{\hat A}} = \gamma_{\mathbb{\hat A}}$ \cite[Section
  2]{GehJon2004}.
\end{proof}

\section{Comparing profinite completion and MacNeille completion}

\begin{thm}[cf.~\protect{\cite[Theorem
    4.12]{BV2007}}] \label{ProfMacN}
  Let $\mathbb{A}$ be a modal algebra. TFAE:
  \begin{enumerate}
  \item the profinite completion $(\mu, \mathbb{\hat A})$ is the
    MacNeille completion of $\mathbb{A}$, and $\Diamond$ is smooth,
  \item $\mathbb{A}$ is residually finite and for every $\theta\in
    \Phi_{\mathbb{A}}$, $1/\theta$ is a principal lattice filter.
  \end{enumerate}
\end{thm}
\begin{proof} 
  If $(\mu, \mathbb{\hat A})$ is the MacNeille completion of
  $\mathbb{A}$, then $\mu\colon \mathbb{A} \rightarrow \mathbb{\hat
    A}$ must be injective, so that $\mathbb{A}$ is residually finite
  (see above). Let $\theta\in \Phi_{\mathbb{A}}$, then it follows from
  the definition of $\mathbb{\hat A}$ that the projection
  $\pi_{\theta} \colon \mathbb{\hat A} \twoheadrightarrow
  \mathbb{A}/\theta$ commutes with $\mu$ and the natural map $a\mapsto
  a/\theta$; i.e.~$a/\theta=\pi_{\theta}\mu(a)$. By \cite[Lemma
  2.7]{BGMM2005}, $\pi_{\theta}^{-1} (1/\theta)$ is a closed principal
  filter of $\mathbb{\hat A}$; say $\pi_{\theta}^{-1}(1/ \theta)=
  \alpha \uparrow$. Because of the correspondence between modal
  filters and modal congruences \cite[Theorem 29]{Venema2007},
  $\alpha \uparrow$ completely characterizes $\mathbb{A}/\theta$ in the
  following sense: $\mathbb{A}/\theta \cong [0,\alpha]_{\mathbb{ \hat
      A}}$ as a bounded lattice \cite[Exercise 4.12]{DP2002}. This
  implies that $\alpha \downarrow$ is finite. Since $(\mu,
  \mathbb{\hat A})$ is the MacNeille completion of $\mathbb{A}$,
  $\mu[\mathbb{A}]$ is join-dense in $\mathbb{\hat A}$, so by
  finiteness of $\alpha \downarrow$, there must exist $a\in
  \mathbb{A}$ such that $\mu(a)=\alpha$. Now $b/\theta = 1/\theta$ iff
  $\mu(b) \in \pi_{\theta}^{-1}(1/\theta) = \alpha \uparrow = \mu(a)
  \uparrow$ iff $b \geq a$, so $1/\theta = a\uparrow$ is a principal
  lattice filter.

  Conversely, if $\mathbb{A}$ is residually finite then $\mu \colon
  \mathbb{A} \rightarrow \mathbb{\hat A}$ is injective, so $(\mu,
  \mathbb{\hat A})$ is a completion of $\mathbb{A}$. To show that
  $(\mu, \mathbb{\hat A})$ is the MacNeille completion of
  $\mathbb{A}$, we will consider the different topologies on
  $\mathbb{\hat A}$. We first show that $\At \mathbb{\hat A} \subseteq
  \mu[\mathbb{A}]$. If $\alpha \in \At \mathbb{\hat A}$ there must be
  some $\theta\in \Phi_{\mathbb{A}}$ such that $\alpha(\theta) \in \At
  \mathbb{A}/\theta$. Because $1/\theta$ is a principal lattice filter
  $c\uparrow$, we know that $\mathbb{A}/\theta \cong
  [0,c]_{\mathbb{A}}$ as a bounded lattice, so there must be some $a
  \leq c$ with $a\in \At \mathbb{A}$ and $a/\theta =
  \alpha(\theta)$. But then $\mu(a)=\alpha$. It follows that $\At
  \mathbb{\hat A} \subseteq \mu[\mathbb{A}]$, whence $\iota \subseteq
  \rho$. Since $\iota$ is Hausdorff, so is $\rho$. Using \cite[Theorem
  4.5]{TheuVen2005} we can thus conclude that, as far as the Boolean
  substructure of $\mathbb{A}$ is concerned, $(\mu, \mathbb{\hat A})$
  is the MacNeille completion of $\mathbb{A}$.  Now to show that
  $\Diamond$ is smooth, remember that $\widehat{\Diamond}\colon
  \mathbb{\hat A} \rightarrow \mathbb{\hat A}$ is $(\tau,
  \tau)$-continuous. Since $\tau = \iota = \gamma$ by Lemma \ref{Tops}
  and $\iota \subseteq \rho$, it follows that $\widehat{\Diamond}$ is
  $(\rho, \gamma)$-continuous. But then since $\gamma^{\downarrow},
  \gamma^{\uparrow} \subseteq \gamma$, it follows by \cite[Proposition
  5.9]{TheuVen2005} that $\Diamond^{\bullet} \leq \widehat{\Diamond}
  \leq \Diamond^{\circ}$. Since also $\Diamond^{\circ} \leq
  \Diamond^{\bullet}$ \cite[Proposition 5.6]{TheuVen2005}, it follows
  that $\Diamond$ is smooth.
\end{proof}
Note that the Theorem above also admits a third equivalent condition,
characterizing the dual space of $\mathbb{A}$. This perspective is
further explored in \cite{BV2007}.

\section{Finitely generated modal algebras with EDPC}

Having equationally definable principal congruences (EDPC) is a strong
meta-logical property of varieties of algebras, that coincides with
e.g.~the existence of a deduction theorem or of a master modality
\cite{BP1994, Kracht1999} for the modal logic corresponding to a
variety of modal algebras. Examples of such logics are logics of
bounded depth, $n$-transitive logics such as $\mathbf{K4}$, or regular
test-free $\mathbf{PDL}$ with finitely many basic programs.

\begin{lem}[Proposition 3.4.3 of
  \protect{\cite{Kracht1999}}] \label{EDPC-principal} Let
  $\mathcal{V}$ be a variety of modal algebras. $\mathcal{V}$ has EDPC
  iff every principal modal filter of an algebra in $\mathcal{V}$ is a
  principal lattice filter.
\end{lem}

Note that the hypotheses below strongly resemble those of
\cite[Theorem 4]{Wolter1997}.

\begin{thm}[cf.~\protect{\cite[Corollary
    4.5.3]{BV2007}}] \label{Thm:FinGen} If $\mathbb{A}$ is a
  residually finite, finitely generated modal algebra such that
  $\operatorname{HSP}(\mathbb{A})$ has equationally definable
  principal congruences, then the profinite completion $(\mu,
  \mathbb{\hat A})$ is the MacNeille completion of $\mathbb{A}$ and
  $\Diamond$ is smooth.
\end{thm}
\begin{proof}
  Since $\mathbb{A}$ is a finitely generated algebra, every $\theta\in
  \Phi_{\mathbb{A}}$ is compact \cite[Theorem 1]{RiSa1978}. As is
  remarked in \cite{Wolter1997}, every compact congruence of a modal
  algebra is principal, so under our hypotheses, every $\theta\in
  \Phi_{\mathbb{A}}$ is principal. Now since
  $\operatorname{HSP}(\mathbb{A})$ has EDPC, Lemma
  \ref{EDPC-principal} tells us that $1/\theta$ is a principal lattice
  filter for all $\theta\in \Phi_{\mathbb{A}}$, so by Theorem
  \ref{ProfMacN}, $(\mu, \mathbb{\hat A})$ is the MacNeille completion
  of $\mathbb{A}$ and $\Diamond$ is smooth.
\end{proof}
Note that the EDPC clause in the Theorem above is suppressed in the
Heyting algebra case \cite{BV2007}, because every variety of Heyting
algebras has EDPC.

The conditions of Theorem \ref{Thm:FinGen} above are sufficient; what
about necessity? From Theorem \ref{ProfMacN} we know that it is
necessary that $\mathbb{A}$ is residually finite. Moreover, in light
of \cite[Section 3.3]{Vosmaer2006}, we know that it is necessary for
Theorem \ref{Thm:FinGen} that $\mathbb{A}$ is atomic. This helps us to
find counterexamples to the Theorem if we remove the requirements of
being finitely generated or having EDPC. For instance, let
$\mathbb{A}$ be the free algebra on $1$ generator for the modal logic
$\mathbf{T}$ (the logic of reflexive Kripke frames). Then
\cite[Corollary 7: Example 1]{Wolter1997} tells us that $\mathbb{A}$
is residually finite and finitely generated but not atomic. This shows
us that being residually finite and finitely generated is not
sufficient for the conclusion of Theorem \ref{Thm:FinGen}.
Alternatively, the free transitive modal algebra on $\omega$
generators is an example of a residually finite modal algebra
$\mathbb{A}$, generating a variety with EDPC, such that $\mathbb{A}$
is not atomic. In summary, residual finiteness is necessary for
Theorem \ref{Thm:FinGen}, and if we remove either of the other two
conditions, we can find non-atomic counterexamples.

\begin{cor}
  Let $\mathbb{A}$ be a finitely generated free algebra for
  $\mathbf{K4}$ or $\mathbf{PDL}$. Then the MacNeille completion and
  the profinite completion of $\mathbb{A}$ are the same and $\Diamond$
  is smooth.
\end{cor}
\begin{proof}
  Let $\mathbf{L}$ be either $\mathbf{K4}$ or $\mathbf{PDL}$ and let
  $\mathcal{V}$ be the variety corresponding to $\mathbf{L}$. Since
  $\mathbf{L}$ has the finite model property, $\mathcal{V}$ is
  generated by its finite members. By \cite[Theorem
  IV-14.4]{Malcev1973}, this implies that $\mathbb{A}$, being a
  finitely generated free algebra for $\mathcal{V}$, is residually
  finite. Using the fact that $\mathbf{L}$ has a master modality, it
  follows that we can apply Theorem \ref{Thm:FinGen}.
\end{proof}


\end{document}